\documentclass[letterpaper,12pt]{amsart}

\usepackage{color,amscd}
\usepackage[draft]{changes}
\usepackage{enumerate}

\usepackage{mathtools}

\definechangesauthor[color=blue]{nc}

\definechangesauthor[color=red]{zl}

\usepackage[abbrev]{amsrefs}
\usepackage{amssymb,amsmath,amsthm,xfrac}

\newtheorem{thm}{Theorem}[section]
\newtheorem{prop}[thm]{Proposition}

\newtheorem{lem}[thm]{Lemma}

\newtheorem{corl}[thm]{Corollary}
\newtheorem{example}[thm]{Example}

\newcommand{\eps}{\varepsilon}

\def \<{\langle}
\def \>{\rangle}

\def \R{\mathbb R}

\def \H{{\cal H}}

\def \H^0{{\cal H}^0 or}

\def \a{{\bf a}}

\def \n{\nabla}
\def \beq{\begin{equation}}
\def \eeq{\end{equation}}

\def \n{\nabla}

\def \eref{\eqref}

\begin{document}



\title{The spectrum of the Laplacian on forms over flat manifolds}

\author{Nelia Charalambous}
\address{Department of Mathematics and Statistics, University of Cyprus, Nicosia, 1678, Cyprus} \email[Nelia Charalambous]{nelia@ucy.ac.cy}

 \author{Zhiqin Lu} \address{Department of
Mathematics, University of California,
Irvine, Irvine, CA 92697, USA} \email[Zhiqin Lu]{zlu@uci.edu}

\thanks{
The first author was partially supported by a University of Cyprus Start-Up grant. The second author is partially supported by the DMS-1510232.}
 \date{\today}

  \subjclass[2010]{Primary: 58J50;
Secondary: 	53C35}

\keywords{essential spectrum, Hodge Laplacian, flat manifolds}

\begin{abstract}
In this article we prove that the spectrum of the Laplacian on $k$-forms over a noncompact flat manifold is always a connected closed interval of the nonnegative real line.  The proof is based on a detailed decomposition of the structure of flat manifolds.
\end{abstract}
\maketitle
\section{Introduction}

We consider the spectrum of the Hodge Laplacian $\Delta$ on differential forms of any order $k$ over a noncompact complete flat manifold $M$. It is well known that the Laplacian is a densely defined, self-adjoint and nonnegative operator on the space of $L^2$ integrable $k$-forms.  The spectrum of the Laplacian consists of all points $\lambda\in \mathbb{R}$   for which $\Delta-\lambda I$ fails to be invertible.  The  essential spectrum consists of the cluster points in the spectrum and of isolated eigenvalues of  infinite multiplicity. We will be denoting the spectrum of the Laplacian on $k$-forms over $M$ by  $\sigma(k,M)$ and its essential spectrum by $\sigma_\textup{ess}(k,M)$.  The complement of the essential spectrum  in $\sigma(k,M)$, which consists of  isolated eigenvalues of finite multiplicity, is called the pure point spectrum and is denoted by $\sigma_\textup{pt}(k,M)$. Since $\Delta$ is nonnegative, its  spectrum  is contained in the nonnegative real line.  The spectrum, the essential spectrum,  and the pure point spectrum  are  closed subsets of  $\mathbb R$.

When the manifold is compact, the essential spectrum is an empty set and the spectrum consists only of discrete eigenvalues. In the case of a noncompact complete manifold, on the other hand, a continuous part in the spectrum might appear.  Unlike the point spectrum, which in most cases cannot be explicitly computed,  the essential spectrum  can be  located either by the classical Weyl criterion as in \cite{Don81}, or by a generalization of it as we have shown in previous work \cite{char-lu-1}. Both criteria require the construction of a large class of test differential forms that act as generalized eigenforms.

Our main goal in this article is to compute the spectrum and essential spectrum of a general noncompact complete flat manifold $M^n$.  The main result of this paper is the following:
\begin{thm} \label{mainthm}
Let $M$ be a flat noncompact complete Riemannian manifold. Then
\[
\sigma(k,M) =\sigma_{\rm ess}(k,M) =  [\alpha, \infty)
\]
for some nonnegative constant $\alpha$.
\end{thm}

The constant $\alpha$ in the above theorem is the first eigenvalue on $\ell$-forms  ($\ell\leq k$) for some compact flat manifold that reflects the structure of $M$ at infinity. In Section \ref{S4} we will give a more precise description of it.

 The structure of a complete flat manifold may be understood through the results of Cheeger-Gromoll,  by which flat manifolds are diffeomorphic to  normal bundles over  totally geodesic compact submanifolds known as the soul.  It is also well known that a complete flat manifold is given as the quotient of Euclidean space  $\mathbb{R}^n$ by a discrete subgroup $\Gamma$ of the Euclidean group $E(n)$. The properties of $\Gamma$, when it is cocompact, are described by the Bieberbach Theorems (see \cites{ES,Wolf}).  We recall  that in order for the quotient to be a complete manifold, $\Gamma$ must be discrete and fixed point free. If $M$ can be decomposed as the direct product of a compact manifold $Z^{n-s}$ times a Euclidean space $\mathbb{R}^s$, then the spectrum of the Laplacian on forms can be computed as we show in Lemma \ref{lem81}. In this  case, the pure point spectrum is empty and the spectrum of the Laplacian on $k$-forms is a closed connected interval as in Theorem \ref{mainthm}.

A general flat manifold however, cannot always be decomposed in product form. Moreover, the classification of non-compact flat manifolds is far from being complete.  Thurston illustrates that the soul of the manifold is essentially the quotient of Euclidean space modulo an abelian subgroup of $\Gamma$, and proves that the manifold is a normal flat bundle over this compact flat submanifold \cite{thurs}.   Mazzeo and Phillips use this decomposition when computing the spectrum of the Laplacian on forms over hyperbolic manifolds in \cite{mazz}. However,  the normal bundle structure of a general flat manifold may be quite complicated, and the soul of the manifold might not provide us with neither sufficient nor useful information when computing the spectrum of the Laplacian on forms (see Example \ref{2.3}).

One of the main accomplishments of this article is to find an adequate subgroup $\Gamma_2$ of the group $\Gamma$, so that  the spectrum of the flat manifold $\mathbb{R}^n/\Gamma$ coincides with that  of $\mathbb{R}^n/\Gamma_2$. Our choice of $\Gamma_2$ will be such that $\mathbb{R}^n/\Gamma_2$ is a product manifold whose spectrum and essential spectrum can be  computed. Finding $\Gamma_2$ is the crux of our results.  Although we begin our analysis of $\Gamma$ by using the subgroup $\Gamma^*$  introduced in J. Wolf \cite{Wolf}, (the same group is also used by Thurston and Mazzeo-Phillips to describe the soul of the manifold \cites{thurs,mazz}) we need to take our decomposition one step further and a more detailed analysis  is necessary.  The details of this construction can be found in Section \ref{S3}.

Over a noncompact manifold the essential spectrum of the Laplacian on functions can present gaps  \cites{Lott2,Post1,Post2,SchoTr}. These manifolds do not have nonnegative Ricci curvature, but some of them  have bounded positive scalar curvature.   On the other hand, it is not uncommon for the spectrum of the Laplacian on forms to have gaps. In Example~\ref{2.1}, over certain even dimensional  quotients of hyperbolic space  $\mathbb{H}^{N+1}$,  for the half dimension $k=(N+1)/2$, the essential spectrum of the Laplacian on $k$-forms is not a connected set. Given all the above known examples,  the result of our paper is somewhat surprising. We show that the spectrum will always be a connected interval  for a  noncompact complete flat manifold.  In a forthcoming paper we will illustrate the significance of this result as it will become essential for  studying the spectrum of the Laplacian on forms over asymptotically flat manifolds.  The latter is a significantly more difficult task as can be seen by comparing the difference between compact flat manifolds and compact almost flat manifolds.

{\bf Acknowledgement.} The authors would like to thank V. Kapovich  and R. Mazzeo for their feedback and useful discussion regarding the structure of flat manifolds. They are also grateful to J. Lott for helping them  work out Example \ref{2.2}.

\section{The Spectrum of a Product Flat Manifold} \label{S2}

Let $M^n=Z^{n-s}\times \mathbb{R}^s$ be the  product of a compact manifold $Z^{n-s}$  of dimension $n-s$ and Euclidean space $\mathbb{R}^s$ of dimnension $s$  for some $1\leq s < n$,  endowed with the product metric.  We denote the Laplace operator on $Z$ by $\Delta_Z$ and on $\mathbb{R}^s$ by $\Delta_{\R^s}$. Since $Z$ is a compact manifold, the spectrum of $\Delta_Z$ on $\ell$-forms  is discrete and we denote by $\lambda_Z(\ell)$ the smallest eigenvalue\footnote{Note for example that $\lambda_Z(0)=0$, not the first nonzero eigenvalue of the compact manifold $Z$.} of $\Delta_Z$ on $\ell$-forms, for $0\leq \ell \leq n-s$.  Recall that  $\lambda_Z(\ell)= \lambda_Z(n-s-\ell)$  by  Poincar\'e duality.

Define $\alpha(Z,s,n,k)=0$ when $s\geq n/2$, or when $s<n/2$  and  either $0\leq k\leq s$, or $n-s\leq k\leq n$. Define $\alpha(Z,s,n,k)=\min\{\lambda_Z(\ell)\mid k-s\leq \ell\leq k\}$, when
$s+1\leq k\leq n/2$, and $\alpha(Z,s,n,k)=\alpha(Z,s,n,n-k)$, when $n/2< k\leq n-s-1$.

\begin{lem} \label{lem81}
Let $M=M^n=Z^{n-s}\times \mathbb{R}^s$ as above. Let $0\leq k\leq n$.
Then $\sigma_{\rm pt}(k, M)=\emptyset$, and
\[
\sigma(k,  M)=\sigma_{\rm ess}(k,  M)=[\alpha(Z,s,n,k),\infty).
\]
\end{lem}

\begin{proof}
By Poincar\'e duality, we know that
\[
\sigma(k,M)=\sigma(n-k,M).
\]
Therefore whenever $s\geq n/2$, or when $s< n/2$ and either $0\leq k\leq s$ or $n-s\leq k\leq n$, we can use the duality to reduce computing the spectrum to the case
$0\leq k\leq s$. Let $x_1,\cdots, x_s$ be the coordinates on $\R^s$. Since $\sigma(0,\R^s)=[0,\infty)$,
it is well-known that for any $\lambda\in \R^+$ and any $\eps>0$, there is a function\footnote{A special choice of $f(x)$ is given in the proof of Proposition~\ref{PropF1}.}  $f\neq 0$ with compact support on $\R^s$ such that
\[
\|\Delta_{\R^s} f-\lambda f\|_{L^2}\leq \eps\|f\|_{L^2}.
\]
Using the test form
\[
\omega=f(x)\,dx_1\wedge\cdots\wedge dx_k
\]
on $\R^s$ and hence on $M=Z\times \R^s$, we have
\[
\|\Delta\omega-\lambda\omega\|_{L^2}\leq C\eps\|\omega\|_{L^2},
\]
where $C>0$ is a constant that only depends on the dimension $n$. It follows that
\[
\sigma(k,M)=[0,\infty).
\]

Now we assume $s< n/2$ and  $s+1\leq k\leq n/2$.  Since  $M=Z\times \R^s$, we have
\[
\Delta_M=\Delta_Z\otimes 1+1\otimes \Delta_{\R^s}.
\]
Therefore,
\[
\sigma(k,M)=\bigcup_{\ell= k-s }^k \left(\sigma(\ell, Z^{n-s})+\sigma(k-\ell,\R^s)\right).
\]
Since $\sigma(k-\ell,\R^s)$  is always $[0,\infty)$, we have
\[
\sigma(k,M)=[\alpha,\infty),
\]
where $\alpha=\alpha(Z,s,n,k)$.
Finally, the case $n/2<k\leq n-s-1$ follows from  Poincar\'e duality.

The result for the  essential spectrum of $M$ follows in a similar manner,  because  $M$ is translation invariant. Since the essential spectrum coincides with the spectrum, the pure point spectrum is an empty set. The translation invariance of $M$ also implies that the point spectrum, the set of eigenvalues of finite multiplicity, is an empty set.
\end{proof}

As we have mentioned, for the computation of the essential spectrum of the Laplacian we must consider the structure of the flat manifold with further detail. The example below illustrates why it is important to not simply consider the manifold as a flat bundle over its soul.
\begin{example}\label{2.3}
Consider $M=\mathbb{R}^3/\Gamma$ where $\Gamma$ is generated by a glide rotation that is composed of an irrational rotation in the $xy$-plane combined with a translation along the $z$ axis. Note that near infinity this flat manifold is essential isometric to $\mathbb{R}^3$, since its injectivity radius becomes infinite.

At the same time, the maximal abelian subgroup $\Gamma^*$  of $\Gamma$ is the translation group $\mathbb{Z}$ and the soul of the manifold is the  circle, $S^1$.  On the other hand, the group $\Gamma_2$ that we describe in the following section, consists only of the identity element. Therefore, the structure of $M$ at infinity is more accurately described by the quotient $\mathbb{R}^3/\Gamma_2$.

One can easily see from this example that  if the flat manifold is not a parallel normal bundle over the soul, then the soul need not determine its spectrum.
\end{example}

When the manifold has negative curvature, the spectrum of the Laplacian on forms can have gaps. This occurs even over hyperbolic space as the following example illustrates.
\begin{example}\label{2.1}
The essential spectrum of the Laplacian on forms over the hyperbolic space $\mathbb{H}^{N+1}$ is given by
\[
\sigma_{\mathrm ess} (k,\Delta) = \sigma_{\mathrm ess} (N+1-k,\Delta) =[\,(\frac{N}{2} - k)^2, \infty\,)
\]
for $0\leq k \leq \frac{N}{2}$, and whenever $N$ is odd
\[
\sigma_{\mathrm ess} (\frac{N+1}{2},\Delta) = \{0\} \cup [\,\frac{1}{4}, \infty\,).
\]
This result can be found in Donnelly~\cite{Don2}. Mazzeo and Phillips also show that the same result is true over hyperbolic manifolds (in other words quotients of hyperbolic space $\mathbb{H}^{N+1}/\Gamma$) that are geometrically finite and have infinite volume \cite{mazz}.
\end{example}

At the same time there was an extensive study of curvature conditions on the manifold so that the essential spectrum of the Laplacian on functions is the nonnegative real line \cites{Esc86,EF92,Lu-Zhou_2011,char-lu-1}. Escobar and Freire also found sufficient curvature assumptions so that the essential spectrum of the Laplacian on forms is $[0,\infty)$. The above Lemma gives however a very simple example of a flat manifold for which the spectrum of the Laplacian on 2-forms is not  $[0,\infty)$.
\begin{example} \label{2.2}
Consider the product manifold $M^4= F^3 \times \mathbb{R}$, where $F^3$ is the compact flat three-manifold constructed by Hantzsche and Wendt in 1935 with first Betti number zero (see \cite{Cobb} for a family of manifolds of any dimension $n\geq 3$ with the same property). Note that the second Betti number of $F^3$ is also zero by Poincar\'e duality. As a result, the first eigenvalue of the Laplacian on 1-forms, $\lambda_{F^3}(1)$, is strictly positive and in fact $\lambda_{F^3}(1)=\lambda_{F^3}(2)>0$.  The product manifold $M$ is a flat noncompact manifold. By  Lemma \ref{lem81}
\[
\sigma_{\mathrm ess} (k,\Delta)=[0,\infty) \qquad \text{for} \ \ k=0,1,3,4.
\]
However, since there do not exist any harmonic 1-forms nor harmonic 2-forms on $F$ then
\[
\sigma_{\mathrm ess} (2,\Delta)=[\lambda_{F^3}(1),\infty) \qquad \text{where}  \qquad \lambda_{F^3}(1)=\lambda_{F^3}(2)>0.
\]
In other words, its essential spectrum is smaller in half-dimension.
\end{example}

This example illustrates the stronger connection of the spectrum of the Laplacian on forms to the topology of the manifold and shows that sufficient conditions so that the form spectrum is $[0,\infty)$ must be stricter than for the case of functions.

\section{A Characterization of Noncompact Flat Manifolds on Large Sets} \label{S3}

We will now consider the general case of complete flat manifolds. Recall that a complete flat manifold is given as the quotient of Euclidean space by a discrete subgroup $\Gamma$ of the Euclidean group $E(n)$. As has already been mentioned, $\Gamma$ must be discrete and fixed point free in order for the quotient space to be a complete manifold.

Fixing a reference point on the Euclidean space each element of $E(n)$ can be uniquely represented by $(g,a)$ where  $g \in O(n)$ and $a \in \mathbb{R}^n$. The action of $(g,a)$ is given by $(g,a) x = gx+a$ for any $x\in\mathbb{R}^n$.  Note that $\R^n$ and hence any subspace of $\R^n$ can be embedded in $E(n)$ by $a\mapsto (1,a)$.  Define the homomorphism $\psi : E(n) \to O(n)$ by  $\psi(g,a) =g$.

We define $\Gamma^*$ to be the intersection of $\Gamma$ with the identity component of the closure of $\Gamma \cdot \mathbb{R}^n$.   By~\cite{Wolf}*{Theorem 3.2.8}, we know that  $\Gamma^*$ is a normal subgroup  $\Gamma$ of finite index and there exists a vector subspace $V \subset \mathbb{R}^n$ and a toral subgroup $T$ of $O(n)$ such that $T$ acts trivially on $V$, $\Gamma^* \subset T\cdot V$ and $\Gamma^*$ is isomorphic to a discrete uniform subgroup of $V$.

With $V$ defined as above we get the orthogonal decomposition
\[
\R^n=V_\perp\oplus V.
\]
Let
\[
\Gamma_1=\{(g,a)\in\Gamma\mid g\,\, {\rm  leaves }\,\, V_\perp\,\, {\rm invariant,  } \,\, i.e. \,\,g|_{V_\perp}=I_{V_\perp}\}
\]
and define
\[
\Gamma^{**}=\Gamma^*\cap\Gamma_1.
\]

\begin{lem} \label{lemF1}
$\Gamma_1$ is a subgroup of $\Gamma$, and $\Gamma^{**}$ is a normal subgroup of $\Gamma_1$ of finite index.  Moreover, $\Gamma^{**}$ is a subgroup of the translation group which acts on $V$; in other words it acts trivially on
$V_\perp$.
\end{lem}

\begin{proof} Let $(g_1,a_1), (g_2,a_2)\in\Gamma_1$. Then
\[
g_1g_2|_{V_\perp}=I_{V_\perp},
\]
and hence $\Gamma_1$ is a subgroup. By ~\cite{Wolf}*{Theorem 3.2.8}, since $\Gamma^*$ is normal in $\Gamma$, then $\Gamma^{**}$ in also normal in $\Gamma_1$.
Moreover, we have
\[
[\Gamma_1,\Gamma^{**}]\leq[\Gamma, \Gamma^*]<\infty.
\]
Finally, by definition, $\psi(\Gamma^{**})$ is the identity matrix. Therefore $\Gamma^{**} \subset \Gamma^* $ is a  translation group acting on $V$.

\end{proof}

Define
\[
V_2={\rm span}\,(\Gamma^{**}),
\]
and let
\[
V=V_1\oplus V_2
\]
be the orthogonal decomposition of $V$. In this way, we write
\begin{equation}\label{decom-4}
\R^n=V_\perp\oplus V_1\oplus V_2.
\end{equation}
Define the subgroup $\Gamma_2$ of $\Gamma_1$ by
\[
\Gamma_2=\{(g,a)\in\Gamma_1\mid g\,\, {\rm  leaves }\,\, V_1\,\, {\rm invariant,  } \,\, i.e. \,\,g|_{V_1}=I_{V_1}\}.
\]
For any vector $a\in\R^n$, we decompose
\[
a=a_\perp+a_{=}=a_\perp+a_1+a_2,
\]
where $a_{\perp}\in V_\perp, a_=\in V, a_1\in V_1$ and $a_2\in V_2$. Let
\begin{equation}\label{decom-51}
d_\perp=\dim V_\perp, \ \ d_1=\dim V_1, \ \ d_2=\dim V_2.
\end{equation}
We also choose a coordinate system
\begin{equation}\label{decom-5}
a=(a^1,\cdots, a^{d_\perp}, a^{d_\perp+1},\cdots, a^{d_\perp+d_1},a^{d_\perp+d_1+1},\cdots, a^n)
\end{equation}
so as to be compatible with the decomposition~\eqref{decom-4}.  Here for $1\leq j\leq n$, $a^j$ is the $j$-th component of $a$.

Let $g\in O(n)$, we use $g_\perp$, $g_=$, $g_1$, and $g_2$ to represent the restrictions of the operator to $V_\perp$, $V$, $V_1$, and $V_2$, respectively. Note that since $g$ is orthogonal, if the restriction is also an orthogonal matrix, then the corresponding space and its orthogonal complement  are invariant spaces of $g$.

\begin{lem} \label{lemF2}
Let $(g,a)\in\Gamma$.  Then there exists a constant $C$ such that
\[
\|a_\perp\|\leq C.
\]
Moreover, if $(g,a)\in\Gamma_1$, then $a_\perp=0$.
\end{lem}

\begin{proof} We first consider the case $(g,a)\in\Gamma_1$, and we write
\[
(g,a)=\left(\begin{pmatrix}1&0\\0&g_=\end{pmatrix},\begin{pmatrix}a_{\perp}\\a_{=}\end{pmatrix}\right),
\]
Since $\Gamma^{**}$ is of finite index in $\Gamma_1$, by  Lemma \ref{lemF1}, we have
\[
(g,a)^N=\left(\begin{pmatrix}1&0\\0&1\end{pmatrix},\begin{pmatrix}Na_{\perp}\\{*}\end{pmatrix}\right)\in\Gamma^{**}\subset\Gamma^{*}.
\]
for some positive integer $N$. By Lemma~\ref{lemF1}, $\Gamma^{**}$ is  a subgroup of the translation group that acts only on $V$, therefore $Na_\perp=0$; hence $a_\perp=0$.

Now we assume that $(g,a)\in\Gamma$. $\Gamma^*$ is normal and of finite index  in $\Gamma$. Therefore, there exists a finite set
$\{(g_\alpha', a_\alpha')\} \subset \Gamma$ such that
\[
(g,a)=(h, b)(g_\alpha', a_\alpha')
\]
for some $\alpha$ and $(h,b)\in \Gamma^*$. By~\cite{Wolf}*{Theorem 3.2.8}, we have
\[
a_\perp=(ha'_\alpha)_\perp +b_\perp=(ha'_\alpha)_\perp.
\]
Since the set $\{a'_\alpha\}$ is finite and $h \in O(n)$, we conclude that $\|a_\perp\|$ is bounded.
\end{proof}

By the same argument  we have
\begin{lem} \label{lemF2b}
Let $(g,a)\in\Gamma_1$. Then there exists a constant $C$ such that
\[
\|a_1\|\leq C.
\]
Moreover, if $(g,a)\in\Gamma_2$, then $a_1=0$.
\end{lem}

\qed

By the above two lemmas, we have the decomposition
\[
\R^n/\Gamma_2=V_\perp\oplus V_1\oplus V_2/\Gamma_2.
\]
For a vector $x\in V_\perp\oplus V_1$ we write
\[
\tilde x=(x_\perp,x_1,0)\in\R^n.
\]

Let $\hat T$ be the closure $\overline{\psi(\Gamma)}$, where $\psi$ is the projection of $E(n)$ to $O(n)$ defined before. By~\cite{Wolf}*{Theorem 3.2.8}, $\hat T$ is a finite extension of the toral group $T$.

\begin{lem} \label{lemF3}
There exists an $x\in V_\perp\oplus V_1$ such that for all $g\in\hat T$,
\begin{enumerate}
\item If $(g\tilde x-\tilde x)_\perp=0$, then $g_\perp=I_{V_\perp}$;
\item If both $(g\tilde x-\tilde x)_\perp=0$ and $(g\tilde x-\tilde x)_1=0$, then both $g_\perp$ and $g_1$ are the identity.
\end{enumerate}
\end{lem}

The points $x$ described in the above lemma are \emph{sufficiently generic} in the sense that one can find them in any open set. The essence of the argument is the observation that $\hat T$ is   a finite extension of a toral group, and as a result it suffices to only consider the representatives of $\hat T$ over $T$. Even though the elements of $T$ may be uncountable, they are easier to control because they are diagonalizable under a fixed coordinate system.

\begin{proof}
We let $p_0,\cdots, p_t\in \hat T$ be the representatives of the group $\hat T/T$. We assume that
\begin{enumerate}
\item $p_0=I$;
\item $p_i|_{V_\perp}=I_{V_\perp}$ and $p_i|_{V_1}=I_{V_1}$ for $1\leq i\leq t_1$;
\item $p_i|_{V_\perp}=I_{V_\perp}$ but  $p_i|_{V_1}\neq I_{V_1}$; moreover, for all $h\in T$, $hp_i$ belongs to neither category (1) nor (2) above, for $t_1<i\leq t_2$;
\item $p_i|_{V_\perp}\neq I_{V_\perp}$; moreover, for all $h\in T$, $hp_i$ belongs to neither category (1), (2) nor (3) above for $t_2<i\leq t$.
\end{enumerate}
Obviously, there exists an $x\in V_\perp\oplus V_1$ such that
\begin{enumerate}
\item $(p_i\tilde x-\tilde x)_\perp\neq 0$ for $t_2<i\leq t$;
\item $(p_i\tilde x-\tilde x)_1\neq 0$ for $t_1< i\leq t_2$.
\end{enumerate}
By continuity, there exists a neighborhood $U$ of $x$ on which the above two conditions are still held. Therefore, without loss of generality, we assume that all the components of $x$ are not zero.

Since $T$ is the toral group, we may assume that under the coordinate system~\eqref{decom-5} for $\R^n$,  its elements can be represented by
\begin{equation}\label{pqr}
\begin{pmatrix}
S_{\theta_1}\\
&\ddots\\
&& S_{\theta_r}\\
&&& I
\end{pmatrix}
\end{equation}
where $S_{\theta_i}= \begin{pmatrix}
\cos\theta_i &\sin\theta_i\\
-\sin\theta_i& \cos\theta_i
\end{pmatrix}$, and  $2r\leq d_\perp$. We write
\[
V_\perp=P_1\oplus P_2\oplus\cdots\oplus P_r\oplus P_{r+1}
\]
according to the above representation. We prove by contradiction. Let $g$ be an element of $\hat T$ such that $(g\tilde x-\tilde x)_\perp=0$.
We assume that  $g=h p_i$ for some $h\in T$ and $i$. Assume that $i>t_2$. Then
 there exists an $1\leq j\leq r+1$ such that $P_j$ is not invariant under $p_i$. If $j\leq r$,  then
\[
((p_i\tilde x)^{2j-1})^2+((p_i\tilde x)^{2j})^2\neq ((\tilde x)^{2j-1})^2+((\tilde x)^{2j})^2.
\]
 As a result, the projection of
$(g\tilde x-\tilde x)_\perp=(hp_i\tilde x-\tilde x )_{\perp}$ to $P_j$ is not zero, which contradicts the assumption.
Similarly, if $P_{r+1}$ is not invariant  under $p_i$, then again $(g\tilde x-\tilde x)_\perp\neq 0$.
This proves the first part of the lemma.

Similarly, we can prove the case when both $(g\tilde x-\tilde x)_\perp$ and $(g\tilde x-\tilde x)_1$ are zero. The lemma is proved.

\end{proof}

\begin{lem}\label{lemF4}
Let $x$ be   as in the above lemma. Then
\[
\lim_{j\to\infty} \inf_{(g,a)\in\Gamma\backslash\Gamma_2}\|(g,a)(j\tilde x)-j\tilde x\| = \infty.
\]
\end{lem}

\begin{proof}
If the lemma is false,  then there exists a subsequence $\{\lambda_j\}$ of positive integers
\[
\lambda_j\to\infty, \ \  (g_j,a_j)\in\Gamma\backslash {\Gamma_2},
\]
such that
\begin{equation} \label{F4e0}
\|(g_j,a_j)(\lambda_j \tilde x)-\lambda_j\tilde x\|\leq C.
\end{equation}
For the same $N$ as in the proof of Lemma \ref{lemF1}, $(g_j,a_j)^N =(h_j,b_j)\in \Gamma^*$. We will consider two cases.

In the first case we assume $(h_j,b_j)\notin \Gamma_1$ for a subsequence of the $(h_j,b_j)$ for which we use the same notation. We have
\begin{equation*}
\|(h_j,b_j)(\lambda_j \tilde x) -\lambda_j \tilde x\| \leq NC
\end{equation*}
since $(g_j,a_j)$ is an isometry. By orthogonality we get
\begin{equation} \label{F4e1}
\|\lambda_j (h_j \tilde x - \tilde x)_\perp + (b_j)_\perp \|\leq NC
\end{equation}
and
\begin{equation} \label{F4e2}
\|\lambda_j (h_j \tilde x - \tilde x)_=  + (b_j)_= \|\leq NC.
\end{equation}

By Lemma \ref{lemF2}, the upper bound \eref{F4e1} is equivalent to
\[
\|\lambda_j (h_j \tilde x - \tilde x)_\perp \|\leq NC.
\]
Since $\lambda_j\to \infty$, we must have
\begin{equation} \label{F4e3}
\| (h_j \tilde x - \tilde x)_\perp \| \to 0
\end{equation}
as $j\to\infty$.
Given that $(h_j,b_j)\in \Gamma^*$ we know $(h_j \tilde x )_= = (\tilde x)_=$. Using \eref{F4e2} we get $\|(b_j)_=\| \leq NC$. Applying Lemma \ref{lemF2}  once again, we obtain the uniform upper bound
\[
\|b_j\|\leq \|(b_j)_\perp\|  + \|(b_j)_=\| \leq C
\]
for some constant $C$.
$\Gamma^*$ as well as $\Gamma$ are discrete groups, therefore there exist only finitely many $(h_j,b_j)$ with $b_j$ bounded. This implies that for sufficiently large  $j$ we have $(h_j \tilde x -\tilde x)_\perp=0$, which is only possible, by Lemma~\ref{lemF3}, when we have $(h_j)_\perp=I_{V_\perp}$. Therefore
$(h_j,b_j)\in \Gamma_1$ and we get a contradiction.

We now consider the case $(h_j,b_j)\in \Gamma_1$ for all sufficiently large $j$. This implies that $(h_j,b_j)\in \Gamma_1\cap \Gamma^* =\Gamma^{**}$. As a result $g_j^N=h_j = I_{\mathbb{R}^n}$. On the other hand, using~\eqref{F4e0}, we get
\begin{equation*}
\|\lambda_j (g_j \tilde x - \tilde x)_\perp \|\leq C
\end{equation*}
for all $j$, which again imples that
\begin{equation}\label{3-est}
\|(g_j \tilde x - \tilde x)_\perp \| \to 0
\end{equation}
as $j\to \infty$.

Since $\Gamma$ is finite over $\Gamma^*$, we can write
\[
(g_j,a_j)=(g_j',a_j')(p_j,q_j)
\]
for  $(p_j,q_j)\in\Gamma^*$ and finitely many $(g_j',a_j')$ . By passing to a subsequence if necessary, we assume that
\[
g_j\to g_\infty, \ \ g_j'\to g_\infty', \ \ p_j\to p_\infty.
\]
By~\eqref{3-est} and Lemma~\ref{lemF3}, we have $g_\infty|_{V_{\perp}}=I_{V_\perp}$.
Since $g_\infty=g_\infty'p_\infty$ and $V_\perp$ is an  invariant space for  $p_\infty$,  it is also an invariant subspace for  $g_\infty'$. However, since there exist only finitely many  $g_j'$, we must have $g_j'=g_\infty'$ for $j\gg 0$. Thus $V_\perp$ is an  invariant space for $g_j'$ for $j\gg 0$. As a result, $V_\perp$ is an invariant space for $g_j=g_j'p_j$ for $j\gg 0$. Since $g_j^N=I_{V_\perp}$ and $g_j$ satisfies ~\eqref{3-est}, we conclude that $g_j=I_{V_\perp}$ for $j\gg 0$.

By Lemma~\ref{lemF2b}, we also must have $(g_\infty \tilde x-\tilde x)_1=0$. Using Lemma \ref{lemF3} we get $g_\infty|_{V_1}=I_{V_1}$, and hence $g_\infty|_{V_\perp \oplus V_1}=I_{V_\perp \oplus V_1}$. Since $g_j^N=I_{\mathbb{R}^n}$ and $g_j\to g_\infty$ with $g_\infty|_{V_\perp \oplus V_1}=I_{V_\perp \oplus V_1}$, we must have $g_j|_{V_\perp \oplus V_1}=I_{V_\perp \oplus V_1}$ for $j\gg 0$. Thus $(g_j,a_j)\in \Gamma_2$, which is a contradiction.

\end{proof}

\begin{thm} \label{thmF1} Let $M=\mathbb{R}^n/\Gamma$ be a flat noncompact Riemannian manifold. Then there exists a compact flat manifold $Z$ of dimension $n-s$, such that for any sufficiently large real number $R>0$, there exists an isometric embedding
\[
Z\times B^{s}(R)\to \R^n/\Gamma,
\]
where $B^{s}(R)$ is a ball of radius $R$ in the Euclidean space $\R^{s}$.
\end{thm}
\begin{proof}
We take $Z=V_2/\Gamma_2$ which is compact given our choice of $\Gamma_2$.  Note that
 $s=d_\perp+d_1$. Choose any $C>0$ such that the diameter of $Z$ satisfies $\mathrm{diam}(Z)\leq C$. By Lemma \ref{lemF4} there exists a $\lambda >0$  such that
\[
\inf_{(g,a)\in\Gamma\backslash\Gamma_2}\|(g,a)(\lambda \tilde x)-\lambda\tilde x\| \geq 3C.
\]
This implies that the ball of radius $C/2$ of $\R^n/\Gamma_2$ can be isometrically embedded into $\R^n/\Gamma$.
\end{proof}

\section{Computation of the Spectrum} \label{S4}

The isometric embedding we have chosen in Theorem \ref{thmF1} will now allow us to compute the spectrum of the Laplacian on $k$-forms over the flat manifold. We will first show that this embedding implies that the essential spectrum contains a connected interval.
\begin{prop} \label{PropF1}
Let $M=\mathbb{R}^n/\Gamma$ be a flat noncompact Riemannian manifold. Then for any $0\leq k\leq n$,
\[
\sigma_{\rm ess} (k,M) \supset \sigma_{\rm ess}  (k, \mathbb{R}^n/\Gamma_2 ).
\]
\end{prop}

\begin{proof}
Note that $\mathbb{R}^n/\Gamma_2=Z \times \mathbb{R}^s$ where $Z=V_2/\Gamma_2$ and $V_2$ is the Euclidean space of dimension $n-s$.  Let $\lambda\in\sigma_{\rm ess}(k, \R^n/\Gamma_2)$. Then by the proof of Lemma~\ref{lem81}, the approximate eigenforms  can be chosen as
\[
\omega_1\wedge\rho\, e^{i\sqrt{\mu_1} r}\,\omega_2
\]
where $\omega_1$ is the first eigenform  corresponding to the smallest eigenvalue $\lambda_1$ of the Laplacian on $\ell$-forms for some $\ell$ on $Z=V_2/\Gamma_2$,  $r$ is the distance function to the origin on $\mathbb R^s$, $\omega_2=dx_1\wedge\cdots\wedge dx_{k-\ell}$, and $\rho$ is the standard cut-off function so that the approximation eigenform is of compact support within the annulus of radii $R-1$ and $2R+1$ for $R\gg 0$.  $\mu_1$ is chosen so that
\[
\lambda=\lambda_1+\mu_1.
\]
We choose $C>4R$ as in the proof of Theorem~\ref{thmF1}.  Then by a standard integral estimate $\rho\,  e^{i\sqrt{\mu_1} r} \omega_1 \wedge\omega_2 $ becomes the approximate eigenform for both $\mathbb R^n/\Gamma_2$ and  $M$ for any nonnegative real number $\mu_1$. This completes the proof of the proposition.
\end{proof}

We will now prove that the two spectra in Proposition \ref{PropF1} are in fact equal. To achieve this, it suffices to show the following result.
\begin{prop} \label{PropF2} Let $\lambda_o (k, X)$ denote the bottom of the Rayleigh quotient for the Hodge Laplacian on $k$-forms over a Riemannian manifold  $X$.

The following inequalities hold
\begin{align*}
&
\lambda_o (k, M) \geq \lambda_o (k, \mathbb{R}^n/\Gamma_2),\\
& \lambda^{\rm ess}_o (k, M) \geq \lambda^{\rm ess}_o (k ,\mathbb{R}^n/\Gamma_2)=\lambda_o (k ,\mathbb{R}^n/\Gamma_2).
\end{align*}
\end{prop}

A key component of the proof of this proposition is the fact that the group $\Gamma$ is of polynomial growth.
\begin{proof}
To simplify notation we denote $\lambda_o (k, M) = \lambda_o$. Then for any $\eps>0$, there exists a $k$-form $\omega \in {\mathcal{C}}_o^\infty (\Lambda^k (M))$ such that
\[
\frac{\int_M |\n \omega|^2}{ \int_M|\omega|^2} \leq \lambda_o +\eps
\]
since the manifold is flat and the Weitzenb\"ock tensor on $k$-forms vanishes.

 We assume without loss of generality that $ \int_M|\omega|^2=1.$

We consider the covering
\[
\mu: \mathbb{R}^n/\Gamma_2 \to \mathbb{R}^n/\Gamma =M.
\]

If $F=\mathrm{supp} \,(\omega)\ $ is the compact support of $\omega$, we let $\{F_j\}$ be the lift of $F$ in  $\mathbb{R}^n/\Gamma_2.$ We fix a point $y\in \mathbb{R}^n/\Gamma_2$ and denote $B_y(R)$ the ball of radius $R$ at $y$ in $\mathbb{R}^n/\Gamma_2$.  Let
\[
\xi(R) = \# \{ F_j \, \big| \, F_j \cap B_y(R) \neq \emptyset \, \}.
\]

Let $D$ be the diameter of $F_j$.    Let $\rho$ be a cut-off function such that
\begin{equation*}
\begin{split}
&  \rho  = 1 \ \ \text{on } B_y(R_1)\\
& \rho = 0  \ \ \text{outside} \ \ B_y(R_2) \\
& |\n \rho|\leq \frac{C}{R_2-R_1},
\end{split}
\end{equation*}
where $R_2>R_1$ are big numbers.

We consider the pull-back of the form $\omega$, $\eta =\mu^* \omega$, on $\mathbb{R}^n/\Gamma_2$. Then the form $\rho \eta$ is compactly supported in on  $\mathbb{R}^n/\Gamma_2$. For any $\eps_0>0$ we have
\[
\int_{\mathbb{R}^n/\Gamma_2} |\n (\rho \eta)|^2 \leq (1+\eps_0) \int_{\mathbb{R}^n/\Gamma_2}  \rho^2 \, |\n  \eta |^2 + \left(1+\frac{1}{\eps_0}\right) \int_{\mathbb{R}^n/\Gamma_2} |\n \rho|^2  \, |\eta|^2.
\]
Since $\mu$ is a local isometry we estimate
\[
\int_{\mathbb{R}^n/\Gamma_2}  \rho^2 \, |\n  \eta |^2 \leq \xi(R_2) \int_M  |\n  \omega |^2 \leq \xi(R_2) (\lambda_o + \eps)
\]
and
\[
\int_{\mathbb{R}^n/\Gamma_2} |\n \rho |^2  \, |\eta|^2 \leq \xi(R_2) \, \frac{C^2}{(R_2-R_1)^2} \; \int_M   | \omega |^2  =  \xi(R_2) \, \frac{C^2}{(R_2-R_1)^2}.
\]
Combining the above we get
\[
\int_{\mathbb{R}^n/\Gamma_2} |\n (\varphi \eta)|^2 \leq \left((1+\eps_0) (\lambda_o + \eps) + \left(1+\frac{1}{\eps_0}\right) \, \frac{C^2}{(R_2-R_1)^2} \right) \;  \xi(R_2).
\]
On the other hand,
\[
\int_{\mathbb{R}^n/\Gamma_2} |\rho \eta|^2 \geq   \xi(R_1-D)  \; \int_{M} |\omega|^2 =  \xi(R_1-D).
\]
Therefore,
\begin{equation*}
\frac{\int_{\mathbb{R}^n/\Gamma_2} \,|  \n(\rho  \eta) |^2}{\int_{\mathbb{R}^n/\Gamma_2} |\rho\eta|^2}
\leq \left((1+\eps_0) (\lambda_o + \eps) + \left(1+\frac{1}{\eps_0}\right) \, \frac{C^2}{(R_2-R_1)^2} \right)\frac{\xi(R_2)}{\xi(R_1-D)}.
\end{equation*}

By the Bishop-Gromov volume comparison theorem, we have
\[
\frac{\int_{\mathbb{R}^n/\Gamma_2} \,|  \n(\rho  \eta) |^2}{\int_{\mathbb{R}^n/\Gamma_2} |\rho\eta|^2}
\leq \left((1+\eps_0) (\lambda_o + \eps) + \left(1+\frac{1}{\eps_0}\right) \, \frac{C^2}{(R_2-R_1)^2} \right)\left(\frac{R_2}{R_1-D}\right)^n
\]
Choosing $R_1,R_2$ sufficiently large with $R_2/R_1\to 1$, and $\eps_0=(R_2-R_1)^{-1} \to 0$,   we obtain
\[
\lambda_o(k,\Delta, \R^n/\Gamma_2)\leq \lambda_o+o(1).
\]
The proposition is proved. A similar method works for the essential spectrum.
\end{proof}

As we mentioned in the proof of Theorem \ref{thmF1}, the quotient $Z=V_2/\Gamma_2$ is compact. The main theorem, Theorem \ref{mainthm} follows by combining this fact with Propositions \ref{PropF1}, \ref{PropF2} and Lemma \ref{lem81}. In fact, the bottom of the essential spectrum is determined by the eigenvalues of the compact space $Z=V_2/\Gamma_2$ depending on the order of the forms. Below we give a more detailed description of the bottom of the spectrum which is a consequence of Lemma  \ref{lem81}.
\begin{thm} \label{thmF2}
Let $M=\mathbb{R}^n/\Gamma$ be a flat noncompact Riemannian manifold. Let $Z$ be the compact flat manifold $Z=V_2/\Gamma_2$ of dimension $n-s$. Then
\[
\sigma(k, {\R^n}/{\Gamma}) =\sigma_{\rm ess}(k, {\R^n}/{\Gamma}) = \sigma_{\rm ess} (k,   Z^{n-s}\times  \R^{s})  = [\alpha(Z,s,n,k), \infty)
\]
where $\alpha(Z,s,n,k)$ is  as in Lemma \ref{lem81}.
\end{thm}

The flatness of the manifold now gives the following immediate corollary.
\begin{corl}
The same result is true for the covariant Laplacian.
\end{corl}


\end{document}